\documentclass[12 pt]{article}
\usepackage{amsmath, amsthm, latexsym}
\usepackage[all]{xy}
\usepackage{amsfonts}
\usepackage{amssymb}
\usepackage{color, soul, mathabx}
\usepackage{graphicx}
\newcommand\norma[1]{\left\lVert#1\right\rVert}
\newcommand\modulo[1]{\left\lvert#1\right\rvert}

\newtheorem{df}{Definition}[section]
\newtheorem{prop}[df]{Proposition}
\newtheorem{teo}[df]{Theorem}

\newtheorem{lema}[df]{Lemma}

\title{The Bishop-Phelps-Bollobás property for operators defined on $c_0$-sum of Euclidean spaces}

\author{Thiago Grando, Mary Lilian Lourenço}
\date{ }

\begin{document}

\maketitle

\begin{abstract}

The main purpose of this paper is to study the Bishop-Phelps-Bollob\'as  property  for operators on $c_0$-sum of euclidean  spaces. We show that the pair $ (c_0\left(\bigoplus^{\infty}_{k=1}\ell^{k}_{2} \right),Y)$  has
 the Bishop-Phelps-Bollob\'as property  for operators (shortly BPBp for operators) whenever  $Y$ is a uniformly convex Banach space.
\end{abstract}

\

\textsc{Keywords:} Banach space, Bishop-Phelps-Bollob\'{a}s theorem, norm attaining operator.

\

\textsc{MSC 2020:} Primary 46B04; Secondary 46B07, 46B20.

%%%%
\section{Introduction}

    In 1961, Bishop and Phelps \cite{BP}, proved that, for any Banach space, the subset of norm attaining functionals is dense in the 
topological dual space. This result is known as Bishop-Phelps theorem. These authors posed the problem of possible extensions of such a result to operators.

 In 1963, Lindenstrauss \cite{Lins}, started the study of extensions of Bishop-Phelps theorem for operators. In full generality there is no
 parallel version of Bishop-Phelps theorem for operators. Motivated by this result, there has been an effort of many authors to study 
some geometric conditions of the Banach spaces $X$ and $Y$ in order to get the Bishop-Phelps theorem for operators.

 In 1970, Bollob\'as \cite{Bol}, proved a ``quantitative" version of the Bishop-Phelps theorem, which stated that, every norm one functional and
 its almost norming points can be approximated by a norm attaining functional and its norm attaining point. 
The result is known nowadays as the Bishop-Phelps-Bollob\'as theorem. 

In 2008,  Acosta, Aron, Garcia and Maestre \cite{AAGM}, dealt with ``quantitative" 
versions of the Bishop-Phelps theorem for operators.
They defined a new notion for a pair of Banach spaces, which is called the Bishop-Phelps-Bollob\'as property for operators, and provided
many notable results. We recommend the surveys \cite{A1} and \cite{D} on the recent progress concerning the Bishop-Phelps-Bollob\'as property.
 
Many references in the field have appeared, among others, \cite{A}, \cite{AR}, \cite{AAGM}, \cite{AL8}, \cite{CK-sum}, \cite{CK}, \cite{CK-studia},
 \cite{CKLM}, \cite{GL}, \cite{Kim}.

 In \cite{AAGM} the authors showed that $(\ell^n_\infty, Y)$ satisfies the Bishop-Phelps-Bollob\'as property for operators for every $n\in \mathbb{N}$, whenever $Y$ is a
  uniformly convex Banach space. They also  raised the question if $(c_0, Y)$ satisfies the BPBp for operators, whenever $Y$ is a uniformly convex Banach space. 
In this sense, S. K. Kim  \cite{Kim}, answered the question in a positive way. 
 More generally, G. Choi and S. K. Kim proved, in \cite{CK-sum}, that ($ c_0\left(\bigoplus^{\infty}_{k=1}X \right), Y)$ has BPBp for operators, 
if  $X$ is uniformly convex Banach space and $Y$  is $\mathbb{C}$-uniformly convex Banach space. Every uniformly convex complex space is $\mathbb{C}$-uniformly convex but 
the converse is not true.

 The purpose of this paper is to show that ($c_0\left(\bigoplus^{\infty}_{k=1}\ell^{k}_{2} \right), Y)$ satisfies the BPBp for operators whenever
 $Y$ is a uniformly convex Banach space.  
 In this sense, we obtained a different result than the one in \cite{CK-sum}, since the $c_0$-sum is composed with different spaces. Notice that the Banach space $c_0\left(\bigoplus^{\infty}_{k=1}\ell^{k}_{2} \right)$, known as $c_0$-sum of the Euclidean $n$-spaces, is not 
isometric to $c_0$.  The importance of such a space is due to the fact that  C. Stegall, in {\cite{stegall}, showed that $\ell_\infty\left ( \bigoplus_{k=1}^\infty \ell_2^k \right )$ does not have the Dunford-Pettis 
property, but its predual, $\ell_1\left ( \bigoplus_{k=1}^\infty \ell_2^k \right )$  does.

Each $x\in c_0\left(\bigoplus^{\infty}_{k=1}\ell^{k}_{2} \right)$ can be represented by $x=\sum^\infty_{n=1}\sum_{k\in I(n)}x_ke_k$, where for 
every $n \in \mathbb{N}$, $I(n)=\{l\in \mathbb{N}:\, s(n-1)+1\leq l\leq s(n) \}$ with $s:\mathbb{N}_0\rightarrow\mathbb{N}_0$ the 
auxiliar function defined by  $s(n)=0$ if $n=0$ and $s(n)=1+2+\ldots +n$ if $n\neq 0$, and $(e_j)$ is the standard basis of
 $ c_0\left(\bigoplus^{\infty}_{k=1}\ell^{k}_{2} \right)$. The norm of $x$ is given by the formula $\norma{x}:=\sup_{n\in \mathbb{N}}\left(\sum_{k\in I(n)}\modulo{x_k}^2\right)^{1/2}$.

\section{Results}

It  will be convenient to recall the following notation. Let $X$ and $Y$ be Banach spaces (over the scalar field $\mathbb{K}=\mathbb{R}$ or $\mathbb{C}$). We denote by $S_X,$ $B_X$, $X^*$, and $\mathcal{L}(X, Y)$, the unit sphere, the closed unit ball, the topological dual space of $X$ and the space of all
bounded linear operators from $X$ into $Y$, respectively. An operator $T \in  \mathcal{L}(X, Y)$ is said to attain its norm at $x_0 \in S_X,$  if $\Vert T \Vert = \Vert T(x_0) \Vert $. Now, we recall a few definitions.

\begin{df}\cite[Definition 1.1]{AAGM}\label{def-BPBP-operadores}
Let $X$ and $Y$ be real or complex Banach spaces.  We say that the pair $(X,Y)$ has the {\it  Bishop-Phelps-Bollob\'{a}s property for operators} (shortly BPBp for operators)
 if given $\varepsilon > 0$, there are $ \eta
(\varepsilon)  > 0$ and $\beta(\varepsilon)>0$ with $\lim_{t\rightarrow 0}\beta(t)=0$ such that for all   $T \in 
S_{\mathcal{L}(X,Y)}$, if $x_0 \in S_{X}$  is such that $\Vert  T x_0 \Vert > 1- \eta
(\varepsilon )$, then there exist  a point   $u_0 \in S_{X}$ and an  operator   $S
\in S_{{\mathcal{L}(X,Y)}}$ that satisfy the following conditions:
$$
\Vert  Su_0 \Vert  =1, \quad  \Vert u_0- x_0 \Vert < \beta(\varepsilon) \quad \text{and} \quad \Vert
S-T \Vert < \varepsilon.
$$
\end{df}

 A Banach space $X$ is {\it uniformly convex} if for every $\varepsilon>0$ there is a $0<\delta<1$ such that for all $x,y\in B_X$ with $\norma{\frac{x+y}{2}}>1-\delta$, we have $\norma{x-y}<\varepsilon$. In this case, the modulus of convexity is given 
by $\delta(\varepsilon)=\inf \left\{1-\norma{\frac{x+y}{2}}: \, x,y\in B_X, \, \norma{x-y}\geq \varepsilon \right\}$. A Banach space $X$ is {\it strictly convex} if $\norma{\frac{x+y}{2}}<1$ whenever $x,y\in S_X$ and $x\neq y$. We remark that uniform convexity implies strict convexity, but the converse is not true.

We remark that in the following results, we  will use similar techniques as can be found in \cite{A}, \cite{AAGM}, \cite{AR}, \cite{CK-sum} and \cite{Kim}. We decided to
 include the proof of these results for the sake of completeness.

\begin{lema}\label{lema4} Let $Y$ be  a strictly convex Banach space and $ T: c_0\left(\bigoplus^{\infty}_{k=1}\ell^{k}_{2} \right)\rightarrow Y$ a bounded linear operator. If $\|T(x)\|=\|T\|$ for some norm one vector $x=\sum^{\infty}_{n=1}\sum_{k\in I(n)}x_ke_k$ , then 
\begin{eqnarray*}
T(e_j)=0 \quad \text{for all} \quad j\in I(k)\quad \text{if} \quad \sum_{i\in I(k)}\modulo{x_i}^2<1.
\end{eqnarray*}
\end{lema}

\begin{proof}
We can assume that $\norma{T}>0$, otherwise nothing is to be proven. Let $x=\sum^{\infty}_{n=1}\sum_{k\in I(n)}x_ke_k$ be an element of $S_{c_0\left( \oplus^{\infty}_{k=1}\ell^{k}_{2} \right)}$ such that $\|T(x)\|=\|T\|$. By the definition of $c_0$-sum, there is $k_0\in \mathbb{N}$ such that $\sum_{i\in \ I(k_0)}\modulo{x_i}^2<1$, where $I(k_0)=\{j, \ldots, l \}$.
Now, we argue by a contradiction, and assume  that  $T(e_{j})\neq0$. 

Let $v:=\left(x_j\pm\left(1-\left(\sum_{i\in I(k_0)}\modulo{x_i}^2\right)^{1/2}\right), x_{j+1},\ldots, x_l\right)\in \ell^{k_0}_2$. Then,
\begin{eqnarray*}
    \norma{v}_2&=&\norma{\left(x_j\pm\left(1-\left(\sum_{i\in I(k_0)}\modulo{x_i}^2\right)^{1/2}\right), x_{j+1},\ldots, x_l\right)}_2\\
    &\leq& \norma{\left(x_j, x_{j+1},\ldots, x_l\right)}_2+\norma{\left(1-\left(\sum_{i\in I(k_0)}\modulo{x_i}^2\right)^{1/2}, 0,\ldots, 0\right)}_2\\
    &=& \left(\sum_{i\in I(k_0)}\modulo{x_i}^2\right)^{1/2}+ 1-\left(\sum_{i\in I(k_0)}\modulo{x_i}^2\right)^{1/2}=1.
\end{eqnarray*}
This  implies that  
\begin{eqnarray*}
\norma{x\pm\left(1-\left(\sum_{i\in I(k_0)}\modulo{x_i}^2\right)^{1/2}\right)e_{j}}&=&\sup_{n\in \mathbb{N}}\left\{\left(\sum_{i\in I(n)\setminus I( k_0)}\modulo{x_i}^2\right)^{1/2}, \norma{v}_2 \right\}\\
&\leq& 1.
\end{eqnarray*}
By assumption $\norma{T(x)}=\norma{T}$, we have naturally  $\norma{T(2x)}=2 \norma{T},$
\begin{eqnarray*}
2\|T\|& \leq &\norma{T\left(x+\left(1-\left(\sum_{i\in I(k_0)}\modulo{x_i}^2\right)^{1/2}\right)e_{j}\right)}\\
&+&\norma{T\left(x-\left(1-\left(\sum_{i\in I(k_0)}\modulo{x_i}^2\right)^{1/2}\right)e_{j}\right)}\\
&\leq& 2\norma{T}.
\end{eqnarray*}
So, $\norma{T\left(x\pm\left(1-\left(\sum_{i\in I(k_0)}\modulo{x_i}^2\right)^{1/2}\right)e_{j}\right)}=\Vert T \Vert$, and \\
$ \frac{T\left(x\pm\left(1-\left(\sum_{i\in I(k_0)}\modulo{x_i}^2\right)^{1/2}\right)e_{j}\right)}{\|T\|} \in S_Y$. Finally, 
\begin{equation*}
 \left\|  \frac{\frac{T\left(x+\left(1-\left(\sum_{i\in I(k_0)}\modulo{x_i}^2\right)^{1/2}\right)e_{j}\right)}{\|T\|}+\frac{T\left(x-\left(1-\left(\sum_{i\in I(k_0)}\modulo{x_i}^2\right)^{1/2}\right)e_{j}\right)}{\|T\|}}{2}\right\|
\end{equation*}

\[=\left\| \frac{2T(x)}{2\|T\|}\right\|=1.
\]
Since $Y$ is strictly convex we get that $T(e_j)=0$. This is a contradiction, as we are assuming that $T(e_j)\neq 0$. By similar argument one shows that $T(e_i)=0$ for each $i\in I(k_0)$.
\end{proof}

Considering the real case, when $Y$ is strictly convex, we prove that if ($ c_0\left(\bigoplus^{\infty}_{k=1}\ell^k_2 \right), Y)$  satisfies BPBp for
 operators then $Y$ is uniformly convex.

\begin{teo}\label{ecec} Let $X$ be the real Banach space $c_0\left( \oplus^{\infty}_{k=1}\ell^{k}_{2} \right)$ and let $Y$ be a strictly convex real Banach space. If $\left(X,Y\right)$ has the BPBp for operators, then $Y$ is uniformly convex.
\end{teo}

\begin{proof} Suppose that $Y$ is not a uniformly convex Banach space.  Then there exists $\varepsilon>0$ and sequences $(y_k), (z_k)\subset S_{Y}$ such that
\begin{equation}\label{neqq}
\displaystyle \lim_{k\rightarrow \infty}\left\| \frac{y_k+z_k}{2}\right\|=1\quad  \text{and}\quad \|y_k-z_k\|>\varepsilon, \,\,  \forall k.
\end{equation}
%Let $(e_k)$ be the canonical basis of $X$.

 For each positive integer $i\in \mathbb{N}$, we define $T_i:X\longrightarrow Y$ by 

\begin{equation*}
T_i(x)=\left(\frac{x_1+x_2}{2}\right)y_i+\left(\frac{x_1-x_2}{2}\right)z_i, \quad x=(x_k)\in X.
\end{equation*}
 For each $i\in \mathbb{N}$ and each  $x\in S_{X}$ we have that
\begin{equation*}
\norma{T_i(x)}\leq \frac{1}{2}(\modulo{x_1+x_2}+\modulo{x_1-x_2})\leq 1.
\end{equation*}
As $ \left\|T_i(e_1+e_2)\right\|=1$, it follows that $\norma{T_i}=1$, for each $i \in \mathbb{N}$. We observe  that, for each $i\in \mathbb{N}$, $\left\| T_i(e_1)\right\|=\left\| \frac{y_i+z_i}{2} \right\|$, thus \label{neqq} $ \Vert T_i (e_1)\Vert$ converges to $1$ when $i \to \infty.$ 
 This fact, togheter with the hypothesis that $(X,Y)$ has the BPBp for operators, guarantees that there are $\eta(\varepsilon)>0$, $\beta(\varepsilon)>0$ with $\lim_{t\rightarrow 0}\beta(t)=0$,  $i_0\in \mathbb{N}$ such that $\|T_{i_0}(e_1)\|>1-\eta \left(\frac{\varepsilon}{2} \right)$, an operator $R\in S_{\mathcal{L}\left(X,Y\right)}$ and a point $u \in S_{X}$ such that 
\begin{equation}
\|R(u)\| =1, \,\,
\|u-e_1\|<\beta \left(\frac{\varepsilon}{2} \right)<1, \,\,
\|R-T_{i_0}\|<\frac{\varepsilon}{2}.
\end{equation}
Then $\left(\sum_{i\in I(k)}\modulo{u_i}^2\right)^{1/2}<1$ for all $k\in\mathbb{N}\setminus \{ 1\}$, and by Lemma \ref{lema4},  
\begin{eqnarray*}
R(e_k)=0,\quad \text{for all}\ \quad k\in\mathbb{N}\setminus \{ 1\}.
\end{eqnarray*}
Therefore, $R(e_1)=R(e_1+e_2)=R(e_1-e_2)$. This implies that
\begin{eqnarray*}
\|y_{i_0 }-z_{i_0}\|&=&\|T_{i_0}(e_1+e_2)-T_{i_0}(e_1-e_2)\|\\
&=&\|T_{i_0}(e_1+e_2)-R(e_1+e_2)+R(e_1-e_2)-T_{i_0}(e_1-e_2)\|\\
&\leq&\|T_{i_0}-R\|\|e_1+e_2\|+\|R-T_{i_0}\|\|e_1-e_2\|\\
&<&\frac{\varepsilon}{2}+\frac{\varepsilon}{2}=\varepsilon.
\end{eqnarray*}
This is a  contradiction, so  $Y$ is a uniformly convex Banach space.

\end{proof}

We need the next lemma to show the main result. In order to state it, let us recall that for $A\subset \mathbb{N}$ (resp. $A\subset \{ 1,\ldots, n\}$) and $X=c_0\left( \oplus^{\infty}_{k=1}\ell^{k}_{2} \right)$ (resp. $X=\ell_\infty\left(\bigoplus^{n}_{k=1}\ell^{k}_{2}\right)$), $P_A: X\rightarrow X$ is a projection on the components in $A$. 

\begin{lema}\label{lema3} Let $F \subset\mathbb{N}$ and $A=\bigcup_{i\in F}I(i)$. Suppose that $0<\varepsilon<1$ and $Y$ is a uniformly convex Banach space with modulus of convexity $\delta(\varepsilon)$. If $T\in S_{\mathcal{L}(c_0\left( \oplus^{\infty}_{k=1}\ell^{k}_{2} \right),Y)}$ satisfy that $\| TP_{A}\|>1-\delta(\varepsilon)$, then $\| T(I-P_{A})\|\leq \varepsilon$. Analogously, if $T\in S_{\mathcal{L}(\ell_\infty\left(\bigoplus^{n}_{k=1}\ell^{k}_{2}\right),Y)}$ and $A\subset \{1, \ldots, n \} $ satisfy $\| TP_{A}\|>1-\delta(\varepsilon)$, then $\| T(I-P_{A})\|\leq \varepsilon$.
\end{lema}
\begin{proof} Let $0<\varepsilon<1$ and $T\in S_{\mathcal{L}(c_0\left( \oplus^{\infty}_{k=1}\ell^{k}_{2} \right),Y)}$ an operator such that $\norma{TP_{A}}>1-\delta(\varepsilon)$. Then there exists $x\in S_{c_0\left( \oplus^{\infty}_{k=1}\ell^{k}_{2} \right)}\cap P_A(c_0\left( \oplus^{\infty}_{k=1}\ell^{k}_{2} \right)) $ such that $\| TP_{A}(x)\|>1-\delta(\varepsilon)$. Fix an element $y=\sum^{\infty}_{n=1}\sum_{k\in I(n)}y_ke_k\in B_{c_0\left( \oplus^{\infty}_{k=1}\ell^{k}_{2} \right)}$ with $\text{supp} \, y\subset \mathbb{N}\setminus A$, then
\begin{eqnarray*}
\left\| x\pm y\right\|&=&\sup_{j\in\mathbb{N}}\left\{\left(\sum_{i\in I(j), j\in F}\modulo{x_i}^2 \right)^{1/2}, \left(\sum_{i\in I(j), j \in \mathbb{N}\setminus F}\modulo{ y_i}^2 \right)^{1/2}\right\}\nonumber\\
&\leq&1.
\end{eqnarray*}
This implies that $\left\| T(x\pm y)\right\|\leq 1$, for every $y\in B_{c_0\left( \oplus^{\infty}_{k=1}\ell^{k}_{2} \right)}$ with $\text{supp}\, y\subset \mathbb{N}\setminus A$. Notice that, for every $z\in B_{c_0\left( \oplus^{\infty}_{k=1}\ell^{k}_{2} \right)}$, the support of the vector $(I-P_A)(z)$ is a subset of $\mathbb{N}\setminus A$ and then, $\left\| T(x)\pm T(I-P_A)(z)\right\|\leq 1$. Moreover,
\begin{eqnarray*}
\left\| \frac{T(x+ (I-P_A)(z))+T(x-(I-P_A)(z))}{2}\right\|&=&\left\| TP_A(x)\right\|\\
&>&1-\delta(\varepsilon).
\end{eqnarray*}
As  $Y$ is uniformly convex Banach space, we conclude that
\begin{equation*}
\left\| T(x+ (I-P_A)(z))-T(x-(I-P_A)(z))\right\|<\varepsilon.
\end{equation*}
That is, $\Vert T(I-P_A)(z)  \Vert < \frac{\varepsilon}{2}<\varepsilon$, whenever $z\in B_X$.
This shows that $\left\| T(I-P_A)\right\|<\varepsilon$.
\end{proof}

In \cite{AAGM}   the authors proved that the pair $(X, Y) $ has the BPBp for
operators when $X$ and $Y$ are both finite-dimensional. The next proposition claims that the pair $(X,Y)$ has the BPBp for a specific finite dimensional space $X$ and  any  uniformly convex space $Y$.
We remark that Proposition \ref{proposicao4} is  similar  to  Theorem 2.4   in \cite{CK-sum} and the proof goes completely analogously to the one of \cite[Theorem 2.4]{CK-sum}, but we give our version for the readers convenience. We need the next lemma to prove Proposition \ref{proposicao4}. We omit the proof because is just modifications of \cite[Lemma 2.3]{CK-sum}.

%\begin{lema}[\cite{AAGM},{Lemma 3.3}]
 %   Let $z_1, \ldots, z_n$ be complex numbers with $\modulo{z_i}\leq 1$ for every $i\in \{1, \ldots, n \}$,  and $\eta>0$ be such that for a convex series $\sum^{n}_{i=1}\alpha_i$, $\hbox{Re}\sum^{n}_{i=1}\alpha_i z_i> 1-\eta$. Then for every $0<\eta'<1$, the set $A=\{i\in \{1, \ldots, n \}: \, \, \hbox{Re} z_i>1-\eta' \}$ satisfies the estimate $\sum_{i\in A}\alpha_i>1-\frac{\eta}{\eta'}$.
%\end{lema}

\begin{lema}\cite[Lemma 2.3]{CK-sum}\label{Lema23} Let $Y$ be a Banach space and $0<\eta<1$ be given. Assume that $T\in S_{\mathcal{L}\left( \ell_\infty\left(\bigoplus^{n}_{k=1}\ell^{k}_{2}\right), Y\right)}$, $y^*\in S_{Y^*}$ and $x\in S_{\ell_\infty\left(\bigoplus^{n}_{k=1}\ell^{k}_{2}\right)}$ satisfy the estimate $y^*(Tx)=\norma{Tx}>1-\eta$. Then for all $0<\eta'<1$, the sets $N=\left\{k\in \{1,\ldots, n\}:\, \, \sum_{j\in I(k)}\modulo{(T^*y^*)(j)}\neq 0  \right\}$ and
\begin{align*}
A=\left\{k\in N: \, \, \mathrm{Re}\sum_{j\in I(k)}(T^*y^*)(j)x(j)>(1-\eta')\sum_{j\in I(k)}\modulo{(T^*y^*)(j)} \right\}
\end{align*}
satisfy the estimate $\sum_{k\in A}\sum_{j\in I(k)}\modulo{(T^*y^*)(j)}>1-\frac{\eta}{\eta'}$. In particular, 
\begin{align*}
 \mathrm{Re}\sum_{k\in A}\sum_{j\in I(k)}(T^*y^*)(j)x(j)>\left( 1-\frac{\eta}{\eta'}\right)(1-\eta').   
\end{align*}
    
\end{lema}

 To prove the next proposition,  we recall that $NA(X,Y)$ is the subset of $\mathcal{L}(X,Y)$ of all norm attaining operators between $X$ and $Y$.

\begin{prop}\cite[Theorem 2.4]{CK-sum}\label{proposicao4} If $Y$ is a 
uniformly convex Banach space with modulus of convexity $\delta(\varepsilon)$, then   $\left(\ell_\infty\left(\bigoplus^{n}_{k=1}\ell^{k}_{2}\right), Y \right)$ has the Bishop-Phelps-Bollob\'as property for operators. 
\end{prop}
\begin{proof}Let $0<\varepsilon<1$. We define $\eta(\varepsilon)=\min\left\{ \frac{\varepsilon}{16}, \delta\left(\frac{\varepsilon}{16}\right), \delta_1\left(\frac{\varepsilon}{2}\right), \ldots, \delta_n\left(\frac{\varepsilon}{2}\right) \right\}$, where  $\delta_k(\varepsilon)$ is the modulus of convexity of the spaces $\ell^{k}_{2}$, for all $k=1, \ldots, n$. Let $T \in S_{\mathcal{L}\left(\ell_\infty\left( \bigoplus^{n}_{k=1}\ell^{k}_{2}\right), Y \right)}$ and $x_0=\sum^{n}_{k=1}\sum_{j\in I(k)}x_0(j)e_j\in S_{\ell_\infty\left(\bigoplus^{n}_{k=1}\ell^{k}_{2}\right)}$ such that 
    \begin{equation*}
    \norma{T(x_0)}>1-\frac{\eta(\varepsilon)^6}{64}.
    \end{equation*}
    Choose, $y^{*}_0\in S_{Y^*}$ such that $y^{*}_0(T(x_0))=\norma{T(x_0)}$ and define the subsets 
   \begin{equation*}
   N=\left\{k\in \{ 1,\ldots, n\}: \, \, \sum_{j\in I(k)}\modulo{(T^*y^*_0)(j)}\neq 0 \right\}
   \end{equation*}
   and 
    \begin{equation*}
    A=\left\{k\in N: \, \mathrm{Re}\sum_{j\in I(k)}(T^*y^*_0)(j)x_0(j)> \left(1-\frac{\eta(\varepsilon)^3}{8}\right)\sum_{j\in I(k)}\modulo{(T^*y^*_0)(j)} \right\}.
    \end{equation*}
    According to Lemma \ref{Lema23}, 
 \begin{eqnarray*}
    \norma{TP_A}\geq \norma{TP_A(x_0)}&\geq& \modulo{y^*_0(TP_A(x_0))}\\
    &=&\modulo{T^*y^*_0(P_A(x_0))}\\
    &\geq& \hbox{Re}\,{T^*y^*_0(P_A(x_0))}\\
    &=&\hbox{Re} \sum_{k\in A}\sum_{j\in I(k)}(T^*y^*_0)(j)x_0(j)\\
    &>&\left(1-\frac{\frac{\eta(\varepsilon)^6}{64}}{\frac{\eta(\varepsilon)^3}{8}} \right)\left(1- \frac{\eta(\varepsilon)^3}{8}\right)\\
    &>&1-\delta\left(\frac{\varepsilon}{16}\right).
\end{eqnarray*}
   And then, Lemma \ref{lema3} implies that $\norma{TP_A-T}<\frac{\varepsilon}{16}$. Now, let \\
   $\Tilde{x}_0:=P_A\left(\sum^{n}_{k=1}\frac{\sum_{j\in I(k)}x_0(j)e_j}{\left(\sum_{j\in I(k)}\modulo{x_0(j)}^2 \right)^{1/2}}  \right)$. Then
\begin{eqnarray*}
       \norma{T(\Tilde{x}_0)}&\geq& \modulo{y^*_0(T\Tilde{x}_0)}\\
       &=&\modulo{(T^*y^*_0(\Tilde{x}_0))}\\
       &\geq& \hbox{Re} \sum_{j\in A}\sum_{j\in I(k)}(T^*y^*_0)(j)\Tilde{x}_0(j)\\
       &>&\sum_{j\in A}\frac{1}{\left(\sum_{j\in I(k)}\modulo{x_0(j)}^2 \right)^{1/2}}\left(1- \frac{\eta(\varepsilon)^3}{8} \right)\sum_{j\in I(k)}\modulo{(T^*y^*_0)(j)}\\
       &\geq& \left(1- \frac{\eta(\varepsilon)^3}{8} \right)\sum_{j\in A}\sum_{j\in I(k)}\modulo{(T^*y^*_0)(j)}\\
       &>&\left(1- \frac{\eta(\varepsilon)^3}{8} \right)\left(1- \frac{\eta(\varepsilon)^3}{8} \right)\\
       &>&1-\frac{\eta(\varepsilon)^3}{4},
   \end{eqnarray*}
   and
   \begin{eqnarray*}
        \left( \sum_{j\in I(k)}\modulo{\Tilde{x}_0(j)-x_0(j)}^2\right)^{1/2}&=&\norma{\frac{\sum_{j\in I(k)}x_0(j)e_j}{\left(\sum_{j\in I(k)}\modulo{x_0(j)}^2 \right)^{1/2}}-\sum_{j\in I(k)}x_0(j)e_j}_2\\
        &=&\modulo{1-\left(\sum_{j\in I(k)}\modulo{x_0(j)}^2 \right)^{1/2}}\\
       &<&\frac{\eta(\varepsilon)^3}{8}, \quad \hbox{for all}\, \, k\in A.
   \end{eqnarray*}
 Choose $y^*_1\in S_Y^*$ such that $y^*_1(T(\Tilde{x}_0))=\norma{T(\Tilde{x}_0)}$. Let $R:P_A\left(\ell_\infty\left(\bigoplus^{n}_{k=1}\ell^{k}_{2}\right)\right)\rightarrow Y$ be the linear bounded operator defined by
 \begin{equation*}
R(z)=TP_A(z)+\eta(\varepsilon)y^*_1(TP_A(z))\frac{T(\Tilde{x}_0)}{\norma{T(\Tilde{x}_0)}}.
 \end{equation*}
As  $P_A\left(\ell_\infty\left(\bigoplus^{n}_{k=1}\ell^{k}_{2}\right)\right)$ is finite dimensional, $\overline{NA\left(P_A\left(\ell_\infty\left(\bigoplus^{n}_{k=1}\ell^{k}_{2}\right)\right), Y\right)}= \\\mathcal{L}\left(P_A\left(\ell_\infty\left(\bigoplus^{n}_{k=1}\ell^{k}_{2}\right)\right), Y\right)$. Then, there is $Q\in NA\left(P_A\left(\ell_\infty\left(\bigoplus^{n}_{k=1}\ell^{k}_{2}\right)\right), Y\right)$ such that $\norma{Q-R}<\frac{\eta(\varepsilon)^3}{4}$. That is, there exist $w_0\in S_{P_A\left(\ell_\infty\left(\bigoplus^{n}_{k=1}\ell^{k}_{2}\right)\right)}$ such that $\norma{Q(w_0)}=\norma{Q}$, $\norma{Q}=\norma{R}$ and $\norma{Q-R}<\frac{\eta(\varepsilon)^3}{4}$. Furthermore, we obtain the following estimate for $\norma{R(\Tilde{x}_0)}$
\begin{eqnarray*}
    \norma{R(\Tilde{x}_0)}&=&\norma{TP_A(\Tilde{x}_0)+\eta(\varepsilon)y^*_1(TP_A(\Tilde{x}_0))\frac{T(\Tilde{x}_0)}{\norma{T(\Tilde{x}_0)}}}\\
    &=&\norma{T(\Tilde{x}_0)+\eta(\varepsilon)y^*_1(T(\Tilde{x}_0))\frac{T(\Tilde{x}_0)}{\norma{T(\Tilde{x}_0)}}}\\
    &=&\norma{T(\Tilde{x}_0)+\eta(\varepsilon)\norma{T(\Tilde{x}_0)}\frac{T(\Tilde{x}_0)}{\norma{T(\Tilde{x}_0)}}}\\
    &=&\norma{T(\Tilde{x}_0)}(1+\eta(\varepsilon))\\
    &>&\left(1-\frac{\eta(\varepsilon)^3}{4} \right)(1+\eta(\varepsilon))\\
    &=&1-\frac{\eta(\varepsilon)^3}{4}+\eta(\varepsilon)\left(1-\frac{\eta(\varepsilon)^3}{4} \right)
\end{eqnarray*}
and 
\begin{eqnarray*}
    \norma{R(\Tilde{x}_0)}\leq \norma{R}&=&\norma{Q}\\
    &=&\norma{Q(w_0)}\\
    &\leq& \norma{Q(w_0)-R(w_0)}+\norma{R(w_0)}\\
    &<& \frac{\eta(\varepsilon)^3}{4}+1+\eta(\varepsilon)\modulo{y^*_1(T(w_0))}.
\end{eqnarray*}    
By composing with an isometry on $P_A\left(\ell_\infty\left(\bigoplus^{n}_{k=1}\ell^{k}_{2}\right)\right)$ if necessary, we may assume that $\modulo{y^*_1(T(w_0))}=\hbox{Re}\,y^*_1(T(w_0))$. Combining the estimates obtained, we see that $\hbox{Re}\,y^*_1(T(w_0))> 1-\eta(\varepsilon)^2$. Thus,
\begin{eqnarray*}
    \hbox{Re} \, y^*_1\left( T\left( \frac{w_0+\Tilde{x_0}}{2}\right) \right)&>&\frac{1}{2}\left(1-\eta(\varepsilon)^2+1-\frac{\eta(\varepsilon)^3}{4} \right)\\
    &=&1-\frac{\eta(\varepsilon)^2+\frac{\eta(\varepsilon)^3}{4}}{2}\\
    &\geq& 1-\eta(\varepsilon)^2.
\end{eqnarray*}
Now, we define the subset 
\begin{align*}
   B=\left\{k\in A: \, 
\hbox{Re}\sum_{j\in I(k)}(T^*y^*_1)(j)\left(\frac{w_0(j)+\Tilde{x_0}(j)}{2}\right) >  \left(1-\eta(\varepsilon)\right)
\sum_{j\in I(k)}\modulo{(T^*y^*_1)(j)} \right\}. 
\end{align*}

Applying Lemma \ref{Lema23} we see that
\begin{eqnarray*}
    \norma{TP_B}&\geq& \norma{TP_B\left(\frac{w_0+\Tilde{x}_0}{2}\right)}\\
    &\geq& \modulo{y^*_1\left(TP_B\left(\frac{w_0+\Tilde{x}_0}{2}\right)\right)}\\
    &=&\modulo{T^*y^*_1\left(P_B\left(\frac{w_0+\Tilde{x}_0}{2}\right)\right)}\\
    &\geq& \hbox{Re} \,{T^*y^*_1\left(P_B\left(\frac{w_0+\Tilde{x}_0}{2}\right)\right)}\\
    &=&\hbox{Re}\sum_{j\in B}\sum_{j\in I(k)}(T^*y^*_1)(j)\left(\frac{w_0(j)+\Tilde{x_0}(j)}{2}\right)\\
    &>&\left(1-\frac{\eta(\varepsilon)^2}{\eta(\varepsilon)} \right)\left(1- \eta(\varepsilon)\right)\\
    &>&1-\delta\left(\frac{\varepsilon}{16}\right).
\end{eqnarray*}
Then, Lemma \ref{lema3} implies that $\norma{TP_B-T}<\frac{\varepsilon}{16}$. Further, for each $k\in B$
\begin{eqnarray*}
    1-\eta(\varepsilon)&<& \hbox{Re}\sum_{j\in I(k)}\frac{(T^*y^*_1)(j)}{\sum_{j\in I(k)}\modulo{(T^*y^*_0)(j)}}\left(\frac{w_0(j)+\Tilde{x}_0(j)}{2}\right)\\
    &\leq&\norma{\sum_{j\in I(k)}\frac{w_0(j)+\Tilde{x}_0(j)}{2}e_j}_2.
   \end{eqnarray*}
That is, $1-\delta_k\left(\frac{\varepsilon}{2} \right)<\norma{\sum_{j\in I(k)}\frac{w_0(j)+\Tilde{x}_0(j)}{2}e_j}_2$, for every $k\in B$. As $\ell^{k}_2$ is uniformly convex with modulus of convexity $\delta_k$,
then 
\begin{equation*}
   \norma{\sum_{j\in I(k)}(w_0(j)-\Tilde{x}_0(j))e_j}_2<\frac{\varepsilon}{2}, \, \, \text{for every} \, \, k\in B.
\end{equation*}
 Now, we define the linear bounded operator $\Tilde{S}: P_A\left(\ell_\infty\left(\bigoplus^{n}_{k=1}\ell^{k}_{2}\right)\right)\rightarrow Y$ by $\Tilde{S}(z)=QP_B(z)+Q(I-P_B)U(z)$, where $U\in B_{\mathcal{L}\left(P_A\left(\ell_\infty\left(\bigoplus^{n}_{k=1}\ell^{k}_{2}\right)\right), P_A\left(\ell_\infty\left(\bigoplus^{n}_{k=1}\ell^{k}_{2}\right)\right) \right)}$ is chosen such that $U\left(E_k\left(\sum_{j\in I(k)}\Tilde{x}_0(j)e_j\right)\right)=E_k\left(\sum_{j\in I(k)}w_0(j)e_j\right)$ for every $k\in A$ and $E_k:\ell^k_2\rightarrow P_A\left(\ell_\infty\left(\bigoplus^{n}_{k=1}\ell^{k}_{2}\right)\right)$ is the $k^{th}$ injection map. Moreover, for each $z=\sum_{k\in A}\sum_{j\in I(k)}z_je_j\in S_{P_A\left(\ell_\infty\left(\bigoplus^{n}_{k=1}\ell^{k}_{2}\right)\right)}$ 
\begin{eqnarray*}
    \rVert\Tilde{S}(z)\lVert=\norma{Q\left(\sum_{k\in B}\sum_{j\in I(k)}z_je_j+\sum_{k\in A\setminus B}\sum_{j\in I(k)}z_jU(e_j) \right)}\leq \norma{Q},
\end{eqnarray*}
so, $\rVert\Tilde{S}\lVert \leq\rVert Q\lVert$. Let $S: \ell_\infty\left(\bigoplus^{n}_{k=1}\ell^{k}_{2}\right) \rightarrow Y$ be the canonical extension of $\frac{\Tilde{S}}{\rVert\Tilde{S}\lVert}$ and define $z_0=\sum_{k\in B}\sum_{j\in I(k)}w_0(j)e_j+\sum_{k\in A\setminus B}\sum_{j\in I(k)}\Tilde{x}_0(j)e_j+\sum_{k\in \{1,\ldots, n\}\setminus A}\sum_{j\in I(k)}x_0(j)e_j\in S_{\ell_\infty\left(\bigoplus^{n}_{k=1}\ell^{k}_{2}\right)}$.
Then
\begin{eqnarray*}
    1&\geq& \norma{S(z_0)}\\
    &=&\frac{1}{\rVert\Tilde{S}\lVert}\norma{\Tilde{S}\left( \sum_{k\in B}\sum_{j\in I(k)}w_0(j)e_j+\sum_{k\in A\setminus B}\sum_{j\in I(k)}\Tilde{x}_0(j)e_j\right)}\\
    &\geq& \frac{1}{\norma{Q}}\norma{Q\left( \sum_{k\in B}\sum_{j\in I(k)}w_0(j)e_j+\sum_{k\in A\setminus B}\sum_{j\in I(k)}w_0(j)e_j\right)}\\
    &=& \frac{\norma{Q(w_0)}}{\norma{Q}}=1.
\end{eqnarray*}
Thus, $S$ attains its norm at $z_0$. Furthermore, if $k\in B$ then
\begin{eqnarray*}
    \left( \sum_{j\in I(k)}\modulo{z_0(j)-x_0(j)}^2\right)^{1/2}&\leq&  \left( \sum_{j\in I(k)}\modulo{w_0(j)-\Tilde{x}_0(j)}^2\right)^{1/2}\\
    &&+ \left( \sum_{j\in I(k)}\modulo{\Tilde{x}_0(j)-x_0(j)}^2\right)^{1/2}\\
    &<&\frac{\varepsilon}{2}+\frac{\eta(\varepsilon)^3}{8}.
\end{eqnarray*}
If $k\in A\setminus B$ then
\begin{eqnarray*}
    \left( \sum_{j\in I(k)}\modulo{z_0(j)-x_0(j)}^2\right)^{1/2}&=& \left( \sum_{j\in I(k)}\modulo{\Tilde{x}_0(j)-x_0(j)}^2\right)^{1/2}\\
    &<&\frac{\eta(\varepsilon)^3}{8}.
\end{eqnarray*}
And, if $k\in \{1,\ldots, n\}\setminus A$ then
$\left( \sum_{j\in I(k)}\modulo{z_0(j)-x_0(j)}^2\right)^{1/2} = 0$. Thus,
\begin{eqnarray*}
    \norma{z_0-x_0}&=&\max\left\{ \left( \sum_{j\in I(k)}\modulo{z_0(j)-x_0(j)}^2\right)^{1/2}: \, k\in \{1,\ldots, n\}\right\}\\
    &<& \frac{\varepsilon}{2}+\frac{\eta(\varepsilon)^3}{8}<\varepsilon.
\end{eqnarray*}
Finally,
\begin{eqnarray*}
    \norma{S-T}&\leq&\norma{S-TP_A}+\norma{TP_A-T}\\
    &=&\norma{\frac{\Tilde{S}}{\rVert\Tilde{S}\lVert}-TP_A}+\norma{TP_A-T}\\
    &\leq&\norma{\frac{\Tilde{S}}{\rVert\Tilde{S}\lVert}-\Tilde{S}}+\norma{\Tilde{S}-Q}+\norma{Q-R}\\
    &&+\norma{R-TP_A}+\norma{TP_A-T}\\
    &<&\modulo{1-\rVert\Tilde{S}\lVert}+\norma{Q(P_B-I)+Q(I-P_B)U}\\
    && +\frac{\eta(\varepsilon)^3}{4}+\eta(\varepsilon)+\frac{\varepsilon}{16}\\
    &\leq& \modulo{1-\norma{R}}+2\norma{Q(I-P_B)} +\frac{\eta(\varepsilon)^3}{4}\\
    && +\eta(\varepsilon)+\frac{\varepsilon}{16}\\
    &<&\frac{\varepsilon}{16}+\eta(\varepsilon)+2\norma{TP_A-Q}+2\norma{TP_A(I-P_B)}\\
    &&+\frac{\eta(\varepsilon)^3}{4}+\eta(\varepsilon)+\frac{\varepsilon}{16}\\
    &=&\frac{\varepsilon}{16}+\eta(\varepsilon)+\frac{\varepsilon}{8}+3\left( \frac{\eta(\varepsilon)^3}{4}+\eta(\varepsilon)\right)+\frac{\varepsilon}{16}<\varepsilon,
    \end{eqnarray*}
and the proof is complete.
\end{proof}

\begin{teo}\label{tbpbp}If $Y$ is a uniformly convex Banach space, then  $(c_0\left( \oplus^{\infty}_{k=1}\ell^{k}_{2} \right), Y )$ has the Bishop-Phelps-Bollob\'as property for operators.  
\end{teo}

\begin{proof}
Given $0<\varepsilon<1$, choose $\eta(\varepsilon)>0$ the positive number in Proposition \ref{proposicao4}. Assume that $T\in S_{\mathcal{L}\left(c_0\left( \oplus^{\infty}_{k=1}\ell^{k}_{2} \right),Y\right)}$ and $x=\sum^{\infty}_{n=1}\sum_{k\in I(n)}x_ke_k\in S_{c_0\left( \oplus^{\infty}_{k=1}\ell^{k}_{2} \right)}$ satisfy $\norma{Tx}>1-\eta(\varepsilon)^2$ and also $\norma{Tx}>1-\delta(\varepsilon)$, where $\delta(\varepsilon)>0$ is the modulus of convexity of $Y$. Since $c_{00}$ is a dense subspace of $c_0\left( \oplus^{\infty}_{k=1}\ell^{k}_{2} \right)$, we
 can choose a vector $u\in S_{c_0\left( \oplus^{\infty}_{k=1}\ell^{k}_{2} \right)}$ with finite support such that $\norma{T(u)}>1-\eta(\varepsilon)^2$, $\norma{T(u)}>1-\delta(\varepsilon)$ and  $\|x-u\|<\varepsilon$. We define $n=\min\{k\in\mathbb{N}: \text{supp}\, u\subset \bigcup^{k}_{j=1}I(j) \}$, and $A=\bigcup^{n}_{k=1}I(k)$. Thus, $\|TP_A\|\geq \|TP_A(u)\|= \|T(u)\|>1-\delta(\varepsilon)$ and $\|TP_A\|>1-\eta(\varepsilon)^2$. According to the   Lemma \ref{lema3}, $\|T(I-P_A)\|\leq \varepsilon$. Now, let $J:\ell_{\infty}\left(\bigoplus^{n}_{k=1}\ell^{k}_{2}\right)\rightarrow c_0\left( \oplus^{\infty}_{k=1}\ell^{k}_{2} \right)$ be the map defined by
\begin{eqnarray*}
 J(w)=\left\{
\begin{array}{ll}
w_j,  &\text{if} \,\, j\in A, \\
0,   &\text{if} \,\, j\in \mathbb{N}\setminus A.
\end{array}
\right.
\end{eqnarray*} 
Then, $\norma{J(w)} =\max_{1\leq k\leq n}\left(\sum_{j\in I(k)}\modulo{w_j}^2\right)^{1/2}=\norma{w}$, for all $w\in \ell_{\infty}\left(\bigoplus^{n}_{k=1}\ell^{k}_{2}\right)$. Let $Q:\ell_{\infty}\left(\bigoplus^{n}_{k=1}\ell^{k}_{2}\right)\rightarrow Y$ be the bounded linear operator defined by $Q(w)=\frac{TP_AJ}{\norma{TP_AJ}}(w)$ and
the vector $z=(z_j)\in \ell_{\infty}\left(\bigoplus^{n}_{k=1}\ell^{k}_{2}\right)$ given by $z_j=u_j$, if $j\in \text{supp}\, u$ and $z_j=0$ if $j\in A\setminus \text{supp}\, u$. It  is  easy to check that $\left\|Q\right\|=\left\|z\right\|=1$. As  $\norma{TP_AJ}\leq 1$, then 
\begin{equation}
\norma{Q(z)}=\norma{\frac{TP_AJ}{\norma{TP_AJ}}(z)}\geq \norma{TP_A(u)}=\norma{T(u)},
\end{equation}
 and thus,  $\Vert Q(z)\Vert >1-\eta(\varepsilon)$. By  Proposition \ref{proposicao4}  the pair  $\left(\ell_{\infty}\left(\bigoplus^{n}_{k=1}\ell^{k}_{2}\right),Y\right)$ has the BPBp for operators, then there are $\beta(\varepsilon)>0$ with $\lim_{t\rightarrow 0}\beta(t)=0$, $\Tilde{R}\in S_{\mathcal{L}\left(\ell_{\infty}\left(\bigoplus^{n}_{k=1}\ell^{k}_{2}\right), Y\right)}$ and $\Tilde{u}\in S_{\ell_{\infty}\left(\bigoplus^{n}_{k=1}\ell^{k}_{2}\right)}$, such that 
\begin{eqnarray}\label{bib}
\|\Tilde{R}(\Tilde{u})\|=1, \quad  \|z-\Tilde{u}\|<\beta(\varepsilon) \quad \text{and} \quad \lVert\Tilde{R}-Q\rVert<\varepsilon.
\end{eqnarray}
Let  $(e_j)$, $(f_j)$ be the canonical basis of $c_0\left( \oplus^{\infty}_{k=1}\ell^{k}_{2} \right)$ and $\ell_{\infty}\left(\bigoplus^{n}_{k=1}\ell^{k}_{2}\right)$, respectively, and  $R:c_0\left( \oplus^{\infty}_{k=1}\ell^{k}_{2} \right)\longrightarrow Y$ be the bounded linear operator given by
 \begin{equation*}
 R(y)= \sum^{\infty}_{k=1}\sum_{j\in I(k)}y_jR(e_j),    
 \end{equation*}
 where
\begin{eqnarray*}
 R(e_j)=\left\{
\begin{array}{ll}
 \Tilde{R}(f_j),  &\text{if}  \,\, j\in A, \\
0,   &\text{if} \,\, j\in \mathbb{N}\setminus A.
\end{array}
\right.
\end{eqnarray*} 
Moreover, consider the vector $v=(v_j)\in c_0\left( \oplus^{\infty}_{k=1}\ell^{k}_{2} \right)$ defined by
\begin{eqnarray*}
 v_j=\left\{
\begin{array}{ll}
\Tilde{u}_j, &\text{if}  \,\, j\in A, \\
x_j,  &\text{if} \,\, j\in \mathbb{N}\setminus A.
\end{array}
\right.
\end{eqnarray*} 
So  $R\in S_{\mathcal{L}\left(c_0\left( \oplus^{\infty}_{k=1}\ell^{k}_{2} \right),Y\right)}$, $v\in S_{c_0\left( \oplus^{\infty}_{k=1}\ell^{k}_{2} \right)}$ and  
$\|R(v)\|=\|\Tilde{R}(\Tilde{u})\|=1.$

It follows that $R$ attains its norm at $v$. Next we will show that $\Vert R -  T\Vert  < \varepsilon.$ We have
\begin{eqnarray}
\left\|R-T \right\|&
\leq&\displaystyle \left\|R-\frac{TP_A}{\|TP_A\|}\right\|+\left\|\frac{TP_A}{\|TP_A\|}-TP_A\right\|+\left\|TP_A-T \right\|\nonumber\\
&=&\displaystyle \left\|\Tilde{R}-\frac{TP_AJ}{\|TP_AJ\|}\right\|+\left\|\frac{TP_A}{\|TP_A\|}-TP_A\right\|+\left\|TP_A-T \right\|\nonumber\\
%&=&\displaystyle \left\|\Tilde{R}-\frac{TP_AJ}{\|TP_AJ\|}\right\|+\left\|TP_A\left(\frac{I}{\norma{TP_A}}-I \right)\right\|+\left\|TP_A-T \right\|\nonumber\\
&=&\left\|\Tilde{R}-\frac{TP_AJ}{\|TP_AJ\|}\right\|+||TP_A|| \modulo{\frac{1}{||TP_A||}-1}+\left\|TP_A-T \right\|\nonumber\\
&=&\displaystyle \lVert\Tilde{R}-Q\rVert+|1-||TP_A|||+\left\|TP_A-T \right\|\nonumber\\
%&=&\displaystyle \left\|\Tilde{R}-Q\right\|+1-\norma{TP_A}+\left\|TP_A-T \right\|\nonumber\\
&<&\displaystyle \varepsilon+1-1+\eta(\varepsilon)^2+\varepsilon\nonumber <3\varepsilon.
\end{eqnarray}
Finally, we show that the vectors $x$ and $v$ are close. Indeed,  
\begin{eqnarray*}
\left\|v-x \right\|=\left\|P_A(v-x) \right\|&=&\max_{1\leq k\leq n}\left(\sum_{j\in I(k)}\modulo{v_j-x_j}^2)\right)^{1/2}\nonumber\\ 
&=&\max_{1\leq k\leq n}\left(\sum_{j\in I(k)}\modulo{\Tilde{u}_j-x_j}^2)\right)^{1/2}\nonumber\\ 
&\leq& \max_{1\leq k\leq n}\left(\sum_{j\in I(k)}\modulo{\Tilde{u}_j-u_j}^2)\right)^{1/2}\\
&&+\max_{1\leq k\leq n}\left(\sum_{j\in I(k)}\modulo{u_j-x_j}^2)\right)^{1/2}\nonumber\\
&\leq&\norma{\Tilde{u}-z}+\norma{u-x}\nonumber < \beta(\varepsilon)+\varepsilon,
\end{eqnarray*}
where $\lim_{t\rightarrow 0}(\beta(t)+t)=0$. Therefore   $\left(c_0\left( \oplus^{\infty}_{k=1}\ell^{k}_{2} \right),Y\right)$ has the BPBp for operators.
\end{proof}

%%%%

\bigskip

\noindent Department of Mathematics\\
Midwestern Paraná State University\\
85.040-167 -- Guarapuava -- PR -- Brazil\\
e-mail: tgrando@unicentro.br

\

\noindent Department of Mathematics\\
University of S\~ao Paulo\\
05315-970 -- S\~ao Paulo -- Brazil.\\
e-mail: mllouren@ime.usp.br

\end{document}